\documentclass[11pt,reqno]{amsart}

\usepackage[utf8]{inputenc}
\usepackage[T1]{fontenc}
\usepackage[english]{babel}

\usepackage[margin=3cm]{geometry}
\usepackage{xcolor}
\usepackage{amsmath}
\usepackage{amssymb}
\usepackage{amsthm}
\usepackage{mathrsfs}
\usepackage{tikz-cd}
\usepackage{enumitem}
\usepackage{caption}
\usepackage{graphicx}
\usepackage{hyperref}

\numberwithin{equation}{section}
\setlist[enumerate]{label=(\alph*)}
\setlist[itemize]{leftmargin=20pt}

\theoremstyle{plain}
\newtheorem{theorem}{Theorem}[subsection]
\newtheorem{proposition}[theorem]{Proposition}
\newtheorem{conjecture}{Conjecture}[section]
\newtheorem{lemma}[theorem]{Lemma}
\newtheorem{corollary}[theorem]{Corollary}
\theoremstyle{definition}
\newtheorem{definition}[theorem]{Definition}

\newtheorem*{notation}{Notation}
\newtheorem{remark}[theorem]{Remark}
\newtheorem{example}[theorem]{Example}

\newcommand{\THH}{\operatorname{THH}}
\newcommand{\TC}{\operatorname{TC}}

\newcommand{\Ext}{\operatorname{Ext}}
\newcommand{\Tor}{\operatorname{Tor}}
\newcommand{\BP}{\mathrm{BP}}
\newcommand{\MU}{\mathrm{MU}}
\newcommand{\KU}{\mathrm{KU}}
\newcommand{\ku}{\mathrm{ku}}
\newcommand{\Id}{\operatorname{id}}
\newcommand\rightthreearrow{%
        \mathrel{\vcenter{\mathsurround0pt
                \ialign{##\crcr
                        \noalign{\nointerlineskip}$\rightarrow$\crcr
                        \noalign{\nointerlineskip}$\rightarrow$\crcr
                        \noalign{\nointerlineskip}$\rightarrow$\crcr
                }%
        }}%
}
\newcommand{\Fun}{\operatorname{Fun}}
\newcommand{\Sp}{\operatorname{Sp}}
\newcommand{\Tot}{\operatorname{Tot}}
\newcommand{\val}{\operatorname{val}}
\newcommand{\CB}{\operatorname{CB}}
\newcommand{\gr}{\operatorname{gr}}
\newcommand{\Eq}{\operatorname{Eq}}
\newcommand{\ol}{\overline}

\title[Integral topological Hochschild homology of $\ku$]{Integral topological Hochschild homology\\of connective complex K-theory}
\author{David Jongwon Lee}
\address{Department of Mathematics, Northwestern University, Evanston, IL, USA}
\email{davidlee@northwestern.edu}

\begin{document}
    \begin{abstract}
        We compute the homotopy groups of $\mathrm{THH}(\mathrm{ku})$ as a $\mathrm{ku}_\ast$-module using the descent spectral sequence for the map $\mathrm{THH}(\mathrm{ku})\to\mathrm{THH}(\mathrm{ku}/\mathrm{MU})$, which is the motivic spectral sequence for $\mathrm{THH}(\mathrm{ku})$ in the sense of Hahn--Raksit--Wilson. We reduce the computation of homotopy groups to the algebra of the universal formal group law, providing a systematic way to compute THH of quotients of MU. We compute the $E_2$-page of the motivic spectral sequence computing $\THH(\ku)$, and we show that it degenerates at the $E_2$-page.
        
    \end{abstract}
    \maketitle
    \tableofcontents
	\section{Introduction}
	\subsection{Motivation and previous works}
	Topological Hochschild homology of a ring spectrum $R$, defined as
	\[
	    \THH(R) = R\otimes_{R\otimes R^{op}}R
	\]
	is an object of interest largely due to its connection to the algebraic K-theory $K(R)$. The Dundas-Goodwillie-McCarthy theorem \cite[Thm. 7.2.2.1]{DGM} states that the difference between topological cyclic homology $\TC(R)$ and $K$-theory $K(R)$ is locally constant, therefore providing a way to compute $K(R)$. Computing $\THH(R)$ is the very first step towards computing $\TC(R)$. See \cite{NS} for a modern construction of $\TC$ from $\THH$.
	
	In this paper, we compute the homotopy groups of $\THH(\ku)$, where $\ku=\tau_{\geq0}\KU$ is the connective cover of the topological K-theory spectrum, in the hope that this would eventually lead to the computation of $K(\ku)$.
	
	The interest in $K(\ku)$ comes from the chromatic redshift philosophy which roughly suggests that the algebraic $K$ theory of a ``chromatic height $n$ theory'' is of ``chromatic height $n+1$''. In this perspective, Ausoni proved that $K(\ku)$ is of chromatic height $2$ in a suitable sense \cite{A-Kku}. Also, according to \cite{BDRR}, the cohomology theory $K(\ku)$ has a geometric meaning in terms of $2$-vector bundles.

    While the primary focus of this paper is $\THH(\ku)$, whose integral homotopy groups were previously unknown, we expect that the methods of this paper are more widely applicable and reduce the computation of $\THH$ of quotients of $\MU$ to mostly algebraic problems of formal group laws. In particular, using this method, it is also possible to recover the previously known homotopy groups of $\THH$ of $\mathbb F_p,\mathbb Z$, and $\ell$ (the Adams summand) and we hope that the method can be used to compute $\THH$ of truncated Brown--Peterson spectra $\BP\langle n\rangle$ in general. See also Section \ref{sec:FQ}.

    \begin{remark}
        The integral homotopy groups of $\THH(\ell)$ were computed by Angeltveit, Hill, and Lawson in \cite{AHL}. Their argument is topological in nature, depending on the fact that the unit map $\ell\to\THH(\ell)$ is a $K(1)$-local equivalence and, as a result, that the mod $p$ homotopy groups of $\THH(\ell)$ are sparse (\cite{MS}, \cite{AR-Hopf}). In this paper, we try to eliminate the use of these topological inputs or any sparsity arguments and reduce the computation of homotopy groups to the algebra of the universal formal group law, so that the method can be generalized to higher heights more easily. 

        We also note that, by \cite[Theorem 1.5]{A-THHku}, one can recover the homotopy groups of $\THH(\ell)$ by taking certain $\mathbb F_p^\times$-fixed points of the homotopy groups of $\THH(\ku)_p^\wedge$.
    \end{remark}

    \begin{remark}
The spectral sequence to be used in the paper is the descent spectral sequence for $\THH(\ku)\to\THH(\ku/\MU)$. This idea is from the work \cite{HW} of Hahn and Wilson, in which they descend along $\THH(\BP\langle n\rangle)\to\THH(\BP\langle n\rangle/\MU)$ to analyze $\TC(\BP\langle n\rangle)$. We shall see that our descent cover $\THH(\ku/\MU)$ has homotopy groups concentrated in even degrees (Prop. \ref{prop:hopf}). Computations of $\THH$ by descending from \emph{even rings} have successfully been carried out by many authors for $\THH$ of ring of integers of $p$-adic number fields and their quotient rings (\cite{BMS2}, \cite{KN}, \cite{LW}). The relation between these works and our work can be explained by the notion of \emph{even/motivic filtration} in the work \cite{HRW} of Jeremy Hahn, Arpon Raksit, and Dylan Wilson. See also Remark \ref{rmk:evenf}. In particular, this paper shows that even for non-discrete rings, motivic spectral sequences are very powerful for computations.
    \end{remark}

    \begin{remark}
        The homotopy groups of various quotients of $\THH(\ku)$ were computed by Ausoni in \cite{A-THHku}. We recover many of them in the course of our computation.

        We also note that several computations related to $\THH(\BP\langle n\rangle)$ have recently been made in the work \cite{AKCH} of Angelini--Knoll, Culver, and Höning.
    \end{remark}
 
	\subsection{Main Results}\label{ssec:main}
	    To describe $\THH_\ast(\ku)$ as a $\ku_\ast=\mathbb Z[\beta]$-module with $|\beta|=2$, we need to define two graded $\mathbb Z[\beta]$-modules $F$ and $T$.
        The graded $\mathbb Z[\beta]$-module $F$ is defined to be
        \[
            F:=\Sigma^3 \left(\mathbb Z\left\{\frac{\beta^k}{f(k)};k\geq 0\right\}\right)\subseteq\Sigma^3(\mathbb Q[\beta])
        \]
        where $f(k)$ is a sequence defined as
        \begin{align*}
            f(0)&=1\\
            f(k) &= \begin{cases}
                pf(k-1)&\text{if $k+2=p^m$ for some prime $p$}\\
                f(k-1)&\text{otherwise.}
            \end{cases}
        \end{align*}
        The graded module $T$ is defined to be the direct sum
        \[
            T := \bigoplus_{p:\mathrm{prime}} T(p)
        \]
        where $T(p)$ is a torsion $\mathbb Z_p[\beta]$-module described below.
        \begin{theorem}\label{thm:main}
            There is an isomorphism
            \[
                \THH_\ast(\ku) = \mathbb Z[\beta]\oplus F\oplus T
            \]
            as $\mathbb Z[\beta]$-modules.
        \end{theorem}
        
        Let us describe the torsion $\mathbb Z_p[\beta]$-module $T(p)$ for each prime $p$. We first define a $\mathbb Z_p[\beta]$-module $T_1(p)$ using generators $h_{m,j}$ of degree $2pm+2$ where $m$ varies over positive integers and $j$ is a nonnegative integer such that $0\leq j\leq \val_p(m)$. If we write $d=\val_p(m),d'=\val_p(m-p^d(p-1))$, the relations are
        \[
        \beta^{p^{d-j+1}-2}h_{m,j}=0
        \]
        and
	    \begin{align*}
	        ph_{m,j} =\begin{cases}
	        h_{m,j+1} + \beta^{p^{d+2}-p^{d+1}} h_{m-p^d(p-1), d'-d-1}&\text{if $j=0$, $m>p^d(p-1)$,}\\
	        &\text{and $d'>d$,}\\
	        h_{m,j+1}&\text{otherwise}.
	        \end{cases}
	    \end{align*}
	    If $j=\val_p(m)$, then the $h_{m,j+1}$'s on the right-hand side should read zero. Then, when $p$ is odd, we define $T(p)$ as $T_1(p)$, and if $p=2$,  then $T(2)$ is defined to be the subquotient of $T_1(2)$ generated by the elements of the form $h_{2m,j}$ with additional relations $h_{2m,\val_2(2m)}=0$ for all positive integers $m$.
        
        Figure \ref{fig:T3} shows the associated graded group of $T(3)$ with respect to the $(3,\beta)$-adic filtration in a range of degrees. The generators $h_{m,0}$ for $6\leq m\leq 17$ are detected by the dots in the bottom row. Compare this with \cite[Figure 1]{AHL}.
        
        \begin{figure}[htbp]
            \centering
            \includegraphics[clip, width=\textwidth, trim=2.1cm 17.6cm 2.1cm 1cm]{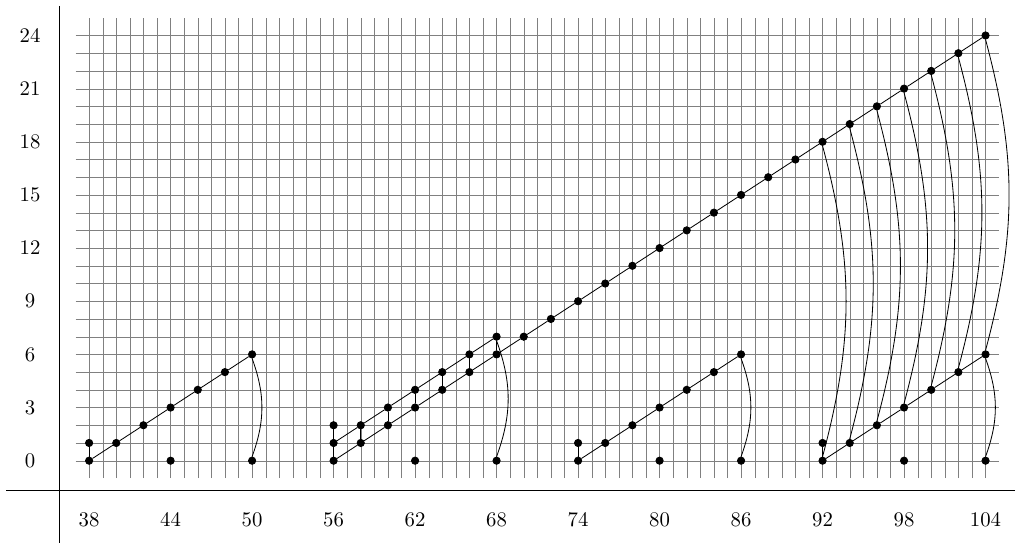}
            \caption{$T(3)$ in degrees $38$ to $104$.}\label{fig:T3}
        \end{figure}
        
        We shall prove Theorem \ref{thm:main} by first computing the $p$-completed homotopy groups of $\THH(\ku)$ for each prime $p$ and then applying the arithmetic fracture square.
        
        For $p=2$, we have $\ku_2^\wedge=\ell_2^\wedge$ so that
        \[
            \THH_\ast(\ku)_2^{\wedge} = \mathbb Z_2[\beta]\oplus F_2^{\wedge}\oplus T(2)
        \]
        by \cite[Thm. 2.6]{AHL}. More precisely, for a positive integer $m$, let
        \[
            m = a_0 2^n+\cdots+a_k 2^{n-k}
        \]
        be the $2$-adic representation of $m$ with $a_1,\dots,a_{k-1}\in\{0,1\}$ and $a_0=a_k=1$. Then, the corresponding element to $h_{2m,j}\in T(2)$ in the notation of \cite[Thm. 2.8]{AHL} is $g_w\in \Sigma^{2a_0p^{n+2}+2(p-1)}T_n$ where $w$ is the string $(a_1,\dots,a_k,0,\dots,0)$ with $j$ trailing zeros.
        
        When $p$ is odd, we shall prove the following theorem in this paper.
	    
	    \begin{theorem}\label{thm:mainp}
	        Let $p$ be an odd prime. The $E_2$-page of the descent spectral sequence for $\THH(\ku)_p^\wedge\to\THH(\ku/\MU)_p^\wedge$ has the following description as graded $(\ku_p^\wedge)_\ast = \mathbb Z_p[\beta]$-modules:
	        \begin{align*}
	            E_2^{0,\ast} &= \mathbb Z_p[\beta]\\
	            E_2^{1,\ast} &= \Sigma F_p^\wedge\\
	            E_2^{2,\ast} &= \Sigma^2 T(p)\\
	            E_2^{r,\ast} &=0\quad\text{if $r\neq0,1,2$}.
	        \end{align*}
	        The spectral sequence degenerates at the $E_2$-page and we have an isomorphism
	        \[
	            \THH_\ast(\ku)_p^\wedge = \mathbb Z_p[\beta]\oplus F_p^\wedge\oplus T(p)
	        \]
	        as $\mathbb Z_p[\beta]$-modules.
	    \end{theorem}

        \subsection{Outline}
        In Section 3, we shall set up the descent spectral sequence for $\THH(\ku)\to\THH(\ku/\MU)$ and identify the $E_1$-page with a cobar complex associated with a comodule of a Hopf algebroid. More precisely, we shall use a cosimplicial resolution
        \[
            \THH(\ku) \to \lim_{\Delta} \THH(\ku/\MU^{\bullet + 1})
        \]
        using the standard Adams-Novikov resolution of $S^0\to\lim_{\Delta}\MU^{\bullet+1}$ as the base of the relative $\THH$. In particular, we shall obtain explicit formulas for the coactions in the cobar complex using the right unit formulas $\MU_\ast\to\MU_\ast\MU$.

        Since everything will be linear over $\ku_\ast=\mathbb Z[\beta]$, we shall compute the cohomology of the cobar complex modulo $(p,\beta)$ and use the Bockstein spectral sequences to obtain the ($p$-complete) cohomology of the cobar complex, i.e. the $E_2$-page. The computations of Bockstein differentials are done in Section 4. We note that previous works (\cite{AHL},\cite{MS},\cite{A-THHku}) also compute Bockstein spectral sequences, but the differentials are arguably easier to compute in our case because we compute the Bockstein spectral sequence of a chain complex rather than of a spectrum.
 
	\subsection{Acknowledgements}
	I would like to thank Jeremy Hahn for suggesting the problem and for helpful conversations in the course of this work. I would also like to thank Ben Antieau, Sanath Devalapurkar, and Ishan Levy for helpful conversations related to the work. I would like to thank an anonymous referee for comments on an earlier draft.
	
	\subsection{Notations and Conventions}
	\begin{itemize}
	    \item Given a nonnegative integer $m$ and a prime number $p$, $\val_p(m)$ is the largest integer such that $p^{\val_p(m)}$ divides $m$. We define $\val_p(0)$ to be $\infty$.
	    \item The word \emph{category} shall always mean $\infty$-category. We use $\otimes$ to denote the smash product of spectra.
	    \item If $A$ is an abelian group, we shall use the same symbol $A$ (instead of $HA$) to denote the Eilenberg--MacLane spectrum of $A$. If $A$ is a commutative ring then the corresponding spectrum is canonically an $\mathbb E_\infty$-ring spectrum.
	\end{itemize}
	
	\section{Preliminaries}
	\subsection{Topological Hochschild Homology}
	Suppose that $S$ is an $\mathbb E_\infty$-ring spectrum and $R$ is an $\mathbb E_1$-$S$-algebra. Then, we define the relative topological Hochschild homology as
	\[
	    \THH(R/S) := R\otimes_{R\otimes_S R^{op}}R.
	\]
	More generally, if $M$ is a $R$-bimodule, or equivalently a $(R\otimes_S R^{op})$-module, then the relative THH with coefficients in $M$ is defined as
	\begin{equation}\label{eq:THHc}
	    \THH(R/S;M) := M\otimes_{R\otimes_SR^{op}}R.
	\end{equation}
	
	In this paper, we shall only consider $\mathbb E_\infty$-ring spectra. More precisely, we shall only consider $\THH(R/S;M)$ when $R$ is an $\mathbb E_\infty$-$S$-algebra and $M$ is an $\mathbb E_\infty$-$R$-algebra. In this case, $\THH(R/S;M)$ naturally has a structure of an $\mathbb E_\infty$-$R$-algebra and we have an equivalence
	\[
	    \THH(R/S;M) = \THH(R/S)\otimes_RM
	\]
	of $\mathbb E_\infty$-$R$-algebras.
	
	Suppose that the following is a commutative diagram in the category of $\mathbb E_\infty$-ring spectra:
	\[
	    \begin{tikzcd}
	        R_2&R_1\ar[l]\ar[r]&R_3\\
	        S_2\ar[u]&S_1\ar[l]\ar[r]\ar[u]&S_3\ar[u].
	    \end{tikzcd}
	\]
	Then, we have an equivalence
	\begin{equation*}
	    \THH(R_2\otimes_{R_1}R_3/S_2\otimes_{S_1}S_3) = \THH(R_2/S_2)\otimes_{\THH(R_1/S_1)}\THH(R_3/S_3).
	\end{equation*}
	This equivalence will be used several times without reference.
	
	\subsection{Suspension Elements}
	Suppose that $S$ is a connective $\mathbb E_\infty$-ring spectrum and $R$ is a connective $\mathbb E_\infty$-$S$-algebra. Let $F=\operatorname{fib}(S\to R)$. We shall also assume that $S_\ast\to R_\ast$ is surjective, so that $F_\ast$ is an ideal of $S_\ast$.
	
	Then, we define the suspension map $\sigma$ to be the composite
	\begin{equation}\label{eq:susdef}
	    \Sigma F = 0\coprod_F 0 \to R \coprod_S R \to R\otimes_S R.
	\end{equation}
	We shall also use $\sigma$ to denote the induced map on the homotopy groups $\pi_\ast(F)\to\pi_{\ast+1}(R\otimes_SR)$.
	\begin{lemma}\label{lem:sus}
	    We have $\sigma(xy)=x\sigma(y)$ for any $x\in S_\ast$ and $y\in F_\ast$. In particular, $\sigma$ annihilates the ideal $(F_\ast)^2$.
	\end{lemma}
	\begin{proof}
	    This follows from the fact that every map in \eqref{eq:susdef} is a map of $S$-modules.
	\end{proof}
	The suspension elements will play an important role as primitive elements in Hopf algebras.
	\begin{proposition}\label{prop:hopf}
	    Let $R$ and $S$ be as above. In addition, assume that the odd homotopy groups of $S$ are zero and that the unit map $S_\ast\to R_\ast$ is a quotient by a regular sequence
	    \[
	        R_\ast = S_\ast / (x_1,x_2,\dots).
	    \]
	    
	    Then, there is a $R_\ast$-Hopf algebra structure on $\THH_\ast(R/S)$ such that
	    \begin{enumerate}
	        \item it is free as a $R_\ast$-module,
	        \item the submodule of primitive elements is the free module generated by the elements $\sigma^2 x_i$'s,
	        \item and for each $k$ and a prime $p$, $(\sigma^2 x_k)^p$ is a $R_\ast$-linear combination of $\sigma^2 x_i$'s modulo $p$.
	    \end{enumerate}
	\end{proposition}
	\begin{proof}
	    Let $T=R\otimes_SR$. The Hopf algebra structure on $\pi_\ast\THH(R/S)=\pi_\ast(R\otimes_TR)$ is given, as usual, by the map
        \[
            R\otimes_TR\xrightarrow{\Id\otimes1\otimes\Id} R\otimes_TR\otimes_TR = (R\otimes_TR)\otimes_R(R\otimes_TR)
        \]
        assuming that $\pi_\ast\THH(R/S)$ is free over $R_\ast$. It is a Hopf algebra rather than just a Hopf algebroid because the left and the right units $R\rightrightarrows R\otimes_TR$ are homotopic. This can be seen, for example, by observing that $\THH(R/S)$ is the colimit of $R$ over the circle $S^1$ in the category of $\mathbb E_\infty$-$S$-algebras and that $S^1$ is a connected space.
        
        By \cite[Prop. 3.6]{Ainfty}, the homotopy groups of $T$ form a graded exterior algebra
	    \[
	        T_\ast = \Lambda_{R_\ast}(\sigma x_1,\sigma x_2,\dots).
	    \]
	    Then, the $E_2$-page of the K\"unneth spectral sequence for $\THH_\ast(R/S)$
	    \[
	        E_2 = \Tor^{T_\ast}(R_\ast,R_\ast)\Rightarrow \THH_\ast(R/S)
	    \]
	    can be computed as a divided power algebra
	    \[
	        E_2=\Gamma_{R_\ast}[\sigma^2 x_1,\sigma^2 x_2,\dots].
	    \]
	    The spectral sequence degenerates at the $E_2$-page since everything is concentrated in even degrees. From the $E_2$-page, we can see that $\THH_\ast(R/S)$ is free as a $R_\ast$-module.

	    Taking Whitehead tower of a spectrum defines a lax symmetric monoidal functor $\tau_{\geq\ast}:\Sp\to \Fun(\mathbb Z^{op},\Sp)$, so that
	    \[
	        \tau_{\geq\ast}R \otimes_{\tau_{\geq\ast}T}\tau_{\geq\ast}R
	    \]
	    filters $\THH(R/S)$. Taking the associated graded group, which is a symmetric monoidal process, we get
	    \[
	        \pi_\ast R\otimes_{\pi_\ast T}\pi_\ast R
	    \]
	    whose homotopy groups are
	    \[
	        \Tor^{T_\ast} (R_\ast,R_\ast).
	    \]
	    Therefore, this is the filtration that constructs the K\"unneth spectral sequence. We then see that the K\"unneth filtration is compatible with the coalgebra structure since the comultiplication is induced by the map
	    \[
	        \tau_{\geq\ast}R\otimes_{\tau_{\geq\ast}T}\tau_{\geq\ast}R\xrightarrow{\Id\otimes1\otimes\Id}\tau_{\geq\ast}R\otimes_{\tau_{\geq\ast}T}\tau_{\geq\ast}R\otimes_{\tau_{\geq\ast}T}\tau_{\geq\ast}R
	    \]
	    of filtered spectra.
	    
	    It follows that the element $\sigma^2 x_k$ is primitive, since $\Delta (\sigma^2 x_k)$ must have K\"unneth filtration $1$, and the only possibility is $\sigma^2 x_k\otimes1 + 1\otimes\sigma^2 x_k$. Conversely, any primitive element in $\THH_\ast(R/S)$ must be detected by a primitive element in the associated graded group. But in the associated graded group, the primitives are linear combinations of $\sigma^2 x_i$'s, since it is a divided power coalgebra. This proves (b).
	    
	    Lastly, the (underived) quotient $(\pi_\ast (R\otimes_TR))/p$ is a $(R_\ast/p)$-Hopf algebra and has a filtration such that the associated graded is a divided power Hopf algebra on $\sigma^2 x_i$'s. Therefore, as in the previous paragraph, the primitive elements form a free $(R_\ast/p)$-module generated by $\sigma^2 x_i$'s. Then, (c) follows since $(\sigma^2 x_k)^p$ is primitive.
	\end{proof}
	
	\begin{example}\label{ex:ha}
	    In this paper, we consider $\THH(\ku/\MU)$ and 
	    \[
	    \MU\otimes_{\THH(\MU)}\MU=\THH(\MU/\MU\otimes\MU).
	    \]
	    In the notation of Definition \ref{def:laz}, $\THH_\ast(\MU/\MU\otimes\MU)$ is a $\MU_\ast$-Hopf algebra with primitive elements $\sigma^2 b_i$'s for $i\geq 1$, and these elements will play an important role. For $\THH_\ast(\ku/\MU)$, we will need only (c) of the previous proposition and not its Hopf algebra structure.
	\end{example}
	
	\section{Descent Spectral Sequence}
	
	\subsection{Descent Spectral Sequence}
	\begin{definition}\label{def:dss}
	    Let $B\to C$ be a map of $\mathbb E_\infty$-ring spectra. Then, we can form the augmented cosimplicial diagram
	    \[
	        B \to C\rightrightarrows C\otimes_BC\rightthreearrow\cdots
	    \]
	    of $\mathbb E_\infty$-ring spectra, called the \emph{descent diagram}. If the above is a limit diagram, then we obtain the Bousfield--Kan spectral sequence
	    \[
	        E_1^{s,t} = \pi_t(C^{\otimes_B(s+1)})\Rightarrow \pi_{t-s}(B).
	    \]
	    More generally, if $M$ is a $B$-module, we have a descent diagram
	    \[
	        M\to M\otimes_BC\rightrightarrows M\otimes_BC\otimes_BC\rightthreearrow\cdots
	    \]
	    and, if it is a limit diagram, a spectral sequence
	    \[
	        E_1^{s,t} = \pi_t(M\otimes_B C^{\otimes_B(s+1)})\Rightarrow \pi_{t-s}(M).
	    \]
	    This will be called the \emph{descent spectral sequence for $M\to M\otimes_BC$ (along $B\to C$)}. We will discuss a sufficient condition for the convergence in Lemma \ref{lem:convergence}.
	\end{definition}
	
	\begin{definition}
	    Suppose that $(k,\Gamma)$ is a graded Hopf algebroid such that $\Gamma$ is flat as a $k$-module and $A$ is a right $\Gamma$-comodule. By the cobar complex $\CB_{\Gamma}(A)$ we mean the chain complex
	    \[
	        A\to A\otimes_k\Gamma\to A\otimes_k\Gamma\otimes_k\Gamma\to\cdots
	    \]
	    where the differentials are the alternating sum of comultiplication maps and the coaction maps. This is bigraded so that $\CB^{s,t}_{\Gamma}(A)$ is the degree $t$ part of $A\otimes_k \Gamma^{\otimes_ks}$. The cohomology of $\CB_{\Gamma}(A)$ is denoted by $\Ext_{\Gamma}(A)$, which is again bigraded. This agrees with the usual definition of $\Ext$. See, for example, \cite[A.1.2.12]{green}.
	\end{definition}
	
	\begin{proposition}\label{prop:adams}
	    In Definition \ref{def:dss}, if $\pi_\ast(C\otimes_BC)$ is flat over $\pi_\ast C$, then $\pi_\ast (C\otimes_BC)$ is a Hopf algebroid over $\pi_\ast C$, $\pi_\ast(M\otimes_BC)$ is a comodule, and we have
	    \begin{align*}
	        E_1 &= \CB_{\pi_\ast(C\otimes_BC)}(\pi_\ast(M\otimes_BC))\\
	        E_2 &= \Ext_{\pi_\ast(C\otimes_BC)} (\pi_\ast(M\otimes_BC)).
	    \end{align*}
	\end{proposition}
	\begin{proof}
	    The identification of $E_1$-page follows from the flatness assumptions and $E_2$-page follows from the definition.
	\end{proof}
	
	\begin{remark}
	    In the previous definition, if $M$ is an $\mathbb E_\infty$-$B$-algebra, then the descent spectral sequence for $M\to M\otimes_BC$ along $B\to C$ is isomorphic to the descent spectral sequence for $M\to M\otimes_BC$ along $M\to M\otimes_BC$. So in this case, the definition is less ambiguous without the phrase ``along $B\to C$''.
	\end{remark}
	
	\begin{lemma}\label{lem:convergence}
	    Let $f:B\to C$ be a $1$-connective map of connective $\mathbb E_\infty$-ring spectra. Then, for any connective $B$-module $M$, the descent spectral sequence for $M\to M\otimes_BC$ converges strongly to $\pi_\ast M$.
	\end{lemma}
	\begin{proof}
	    Since the descent spectral sequence is constructed from the coskeletal filtration, the lemma follows from the fact that the fiber of the map
	    \[
	        M\to \operatorname{cosk}^s (M\otimes_B C^{\otimes_B(\bullet+1)})
	    \]
	    is $M\otimes_B I^{\otimes_B(s+1)}$ where $I$ is the fiber of $B\to C$. See, for example, \cite[Prop. 2.14]{MNN}.
	\end{proof}
	
	\begin{lemma}\label{lem:desus}
	    Suppose that a map of connective $\mathbb E_\infty$-ring spectra $f:B\to C$ induces a surjection on homotopy groups. Let $x\in F_\ast$ where $F=\operatorname{fib}(f)$. Then, $\sigma x \in \pi_\ast (C\otimes_BC)$, considered as an element in the $E_1$-page of the descent spectral sequence for $f$, is a permanent cycle and detects $x\in B$.
	\end{lemma}
	\begin{proof}
	    It is enough to show that $x$ is detected by $\sigma x$ in the equalizer
	    \[
	        \Eq(C\rightrightarrows C\otimes_BC).
	    \]
	    This follows from chasing the following diagram
	    \[
	        \begin{tikzcd}
	            F\ar[r,equal]\ar[d]&\Eq(0\rightrightarrows 0\sqcup_F0)\ar[d]&\\
	            B\ar[r]&\Eq(C\rightrightarrows C\sqcup_BC)\ar[r]&\Eq(C\rightrightarrows C\otimes_BC)
	        \end{tikzcd}
	    \]
	    from $x\in F_\ast$.
	\end{proof}
	
	\subsection{Complex K-theory spectrum}
	Until the end of section 4, we shall assume that $p$ is a fixed odd prime number and that every spectrum is $p$-complete. For example, we shall write $\THH(\ku)$ instead of $\THH(\ku_p^\wedge)_p^\wedge$.
	
	We write $\ku$ for the connective cover of the complex K-theory spectrum $\KU$ equipped with the standard complex orientation $\MU\to\ku$, which can be lifted to be an $\mathbb E_\infty$ orientation according to \cite{JK}.
	
	We shall compute the descent spectral sequence for
	\[
	    \THH(\ku)\to\THH(\ku/\MU)=\THH(\ku)\otimes_{\THH(\MU)}\MU
	\]
	along $\THH(\MU)\to \MU$. More generally, $M=\mathbb F_p$ or $\mathbb Z_p$ with a canonical $\ku$-algebra structure, we shall compute the descent spectral sequence for
	\[
	\THH(\ku;M)\to\THH(\ku/\MU;M).
	\]
	Let $E_r^{\ast,\ast}(\THH(\ku;M))$ denote the $E_r$-page of this descent spectral sequence. Until the end of Section 4, let us write
	\begin{align*}
		\Gamma &:= \MU\otimes_{\THH(\MU)}\MU=\THH(\MU/\MU\otimes\MU)\\
		A&:= \THH(\ku/\MU).
	\end{align*}
	Using Propositions \ref{prop:hopf} and \ref{prop:adams}, we note a few things about the homotopy groups of these ring spectra.
	\begin{itemize}
	    \item $A_\ast$ and $\Gamma_\ast$ are commutative even-graded rings.
	    \item $\Gamma_\ast$ is a $\MU_\ast$-Hopf algebra and $A_\ast$ is a right $\Gamma_\ast$-comodule algebra. We shall write $\eta_R:A_\ast\to A_\ast\otimes\Gamma_\ast$ for the coaction map for reasons to be explained in Proposition \ref{prop:dss2}.
	    \item  $E_1^{\ast,\ast}(\THH(\ku))$ can be identified with the cobar complex
    	\[
    	     \CB_{\Gamma_\ast}(A_\ast)= (A_\ast \xrightarrow{D^0} A_\ast\otimes_{\MU_\ast}\Gamma_\ast\xrightarrow{D^1} A_\ast\otimes_{\MU_\ast}\Gamma_\ast\otimes_{\MU_\ast}\Gamma_\ast\xrightarrow{D^2}\cdots),
    	\]
    	where we write $D^0, D^1, \dots$ for the differentials in $\CB_{\Gamma_\ast}(A_\ast)$. These are the $d_1$ differentials in the descent spectral sequence.
	\end{itemize}
	
	\begin{remark}\label{rmk:evenf}
	     Suppose $B\to C$ is a map of $\mathbb E_\infty$-ring spectra. Instead of the coskeletal filtration of the descent diagram for $B\to C$, there is an alternative filtration using Whitehead covers. The $k$'th filtration of $B$ is given by
	    \[
	        \Tot(\tau_{\geq 2k} C^{\otimes_B (\bullet+1)}).
	    \]
	    If we further assume that $C^{\otimes_B(s+1)}$ has no odd homotopy groups for all $s\geq0$, then this filtration gives us a shearing of the descent spectral sequence in the sense that we have
	    \[
	        ^\backprime E_r^{k,3k-s} = E_{2r+1}^{s,2k}
	    \]
	    where the left-hand side is the spectral sequence associated with the new filtration and the right-hand side is the descent spectral sequence. More precisely, they are related by the d\'ecalage \cite{decalage} and doubling the speed of the filtration.
	
	    In our case of $B=\THH(\ku)$ and $C=\THH(\ku/\MU)$, the new filtration using Whitehead covers is an example of the \emph{even/motivic filtration} of Hahn, Raksit, and Wilson \cite[Def. 4.1.2]{HRW} since $C$ is \emph{evenly free} over $B$ in their sense.
	\end{remark}
	
	\begin{lemma}
	    $\CB_{\Gamma_\ast}(A_\ast)$ is a chain complex of free $\ku_\ast$-modules. Therefore, we have identifications
	    \begin{align*}
	        E_1(\THH(\ku;\mathbb Z_p)) &= \CB_{\Gamma_\ast}(A_\ast/\beta)\\
	        E_2(\THH(\ku;\mathbb Z_p)) &= \Ext_{\Gamma_\ast}(A_\ast/\beta)\\
	        E_1(\THH(\ku;\mathbb F_p)) &= \CB_{\Gamma_\ast}(A_\ast/(p,\beta))\\
	        E_2(\THH(\ku;\mathbb F_p)) &= \Ext_{\Gamma_\ast}(A_\ast/(p,\beta)).
	    \end{align*}
	\end{lemma}
	\begin{proof}
	    From Example \ref{ex:ha}, $A_\ast$ is free as a $\ku_\ast$-module and $\Gamma_\ast$ is free as a $\MU_\ast$-module.
	\end{proof}
	
	\begin{notation}
	    From this point, all ordinary modules or ordinary rings will naturally be modules or algebras over $\MU_\ast$, and all tensor product $\otimes$ will be over $\MU_\ast$ unless the base ring is explicitly written. We shall continue to write the base for the tensor product of spectra unless it is over the sphere spectrum.
	\end{notation}

	There is an alternative description of the descent spectral sequence. The augmented cosimplicial diagram
	\[
		S^0 \to \MU\rightrightarrows \MU^{\otimes2} \rightthreearrow \cdots
	\]
	induces an augmented cosimplicial diagram
	\begin{equation}\label{eq:dcs2}
		\THH(\ku)\to\THH(\ku/\MU)\rightrightarrows \THH(\ku/\MU^{\otimes2}) \rightthreearrow \cdots
	\end{equation}
	of $\mathbb E_\infty$-ring spectra. The following proposition is immediate from the definitions.
	
	\begin{proposition}\label{prop:dss2}
	    The augmented cosimplicial diagram \eqref{eq:dcs2} is equivalent to the descent diagram for $\THH(\ku)\to\THH(\ku/\MU)$. Furthermore, under the identification $\THH_\ast(\ku/\MU^{\otimes2})=A_\ast\otimes\Gamma_\ast$ the two maps
	    \[
	    \eta_L,\eta_R:\THH_\ast(\ku/\MU)\to\THH_\ast(\ku/\MU^{\otimes2})
	    \]
	    induced by the left and right units $\eta_L,\eta_R:\MU\to \MU^{\otimes2}$ can be identified with $\Id\otimes1$ and the coaction map $A_\ast\to A_\ast\otimes\Gamma_\ast$, respectively.
	\end{proposition}
	
	\begin{definition}\label{def:laz}
	    Following the classical notation, we write $x_1,x_2,\dots$ with $|x_i|=2i$ for the polynomial generators of the Lazard ring $\MU_\ast$, and we write $b_1,b_2,\dots$ for the generators of $\MU_\ast\MU = \MU_\ast[b_1,b_2,\dots]$ as a $\MU_\ast$-algebra where $\MU_\ast\MU$ is given the algebra structure by the left unit $\MU_\ast\to\MU_\ast\MU$ and $b_i$'s vanish under the multiplication map $\MU_\ast\MU\to\MU_\ast$. There are many choices for the generators, and for now, we only require that $x_1$ maps to $\beta$ under
	    \[
	        \MU_\ast\to \ku_\ast=\mathbb Z_p[\beta]
	    \]
	    and that $x_i$ maps to zero for $i\geq 2$. We shall give more specific choices of generators in Lemma \ref{lem:coaction}.
	    
	    For the lightness of notations, we shall often write $v_k$ instead of $x_{p^k-1}$ when $k\geq 2$ and $t_k$ instead of $b_{p^k-1}$ when $k\geq 1$. Note that we do not write $v_1$ for $x_{p-1}$ because $x_{p-1}$ maps to $0$ in $\ku_\ast$, while it is more natural that a class named $v_1$ maps to a nonzero class in $\ku_\ast$.
	\end{definition}
	
	\begin{notation}
	    The elements $x_i\in\MU_\ast$ for $i\geq 2$ and $b_i\in\MU_\ast\MU$ for $i\geq 1$ admit double suspensions $\sigma^2 x_i\in A_\ast$ and $\sigma^2 b_i\in \Gamma_\ast$. In the cobar complex $\CB_{\Gamma_\ast}(A_\ast)$, we shall write $\sigma^2 x_i$ (or $\sigma^2 v_i$) for the corresponding element in either $A_\ast$ or $A_\ast\otimes\Gamma_\ast$. Which element the notation is referring to will be clear from the context. Similarly, we shall write $\sigma^2 b_i$ (or $\sigma^2 t_i$) for the corresponding element in $A_\ast\otimes\Gamma_\ast$. We will not need any notation for elements in $A_\ast\otimes\Gamma_\ast^{\otimes s}$ for $s\geq 2$.
	\end{notation}
	
	The following two remarks hold for any choice of generators.
	
	\begin{remark}\label{rmk:desus}
	    Consider the descent spectral sequence for $\THH(\MU)\to\THH(\MU/\MU)$, whose $E_1$-page is $\CB_{\Gamma_\ast}(\MU_\ast)$. Since
	    \[
	        \THH_\ast(\MU) = \Lambda_{\MU_\ast}(\sigma b_1,\sigma b_2,\dots),
	    \]
	    we can see, by Lemma \ref{lem:desus}, that the element $\sigma^2 b_i\in \CB_{\Gamma_\ast}^1(\MU_\ast)$ is a permanent cycle in the descent spectral sequence for any $i$ and that it detects $\sigma b_i\in\THH_\ast(\MU)$.
	    
	    Mapping to the descent spectral sequence for $\THH(\ku)\to\THH(\ku/\MU)$, we can see that $\sigma^2 b_i\in \CB^1_{\Gamma_\ast}(A_\ast)$ is a permanent cycle in this descent spectral sequence and detects $\sigma b_i\in\THH_\ast(\ku)$. Here, $\sigma b_i$ is the suspension of the class $b_i\in\ku_\ast\ku$, which is defined to be the image of $b_i\in\MU_\ast\MU$.
	\end{remark}
	
	\begin{remark}\label{rmk:mul}
	    There is a multiplicative structure on the descent spectral sequence, which is represented in $E_1=\CB_{\Gamma_\ast}(A_\ast)$ by the standard formula for cup product of cocycles. We shall only be interested in the multiplication by $\sigma^2 b_i\in A_\ast\otimes\Gamma_\ast$, which is a permanent cycle by the previous remark. In this case, we can check that the cup product formula in the $E_1$-page for $x\in A_\ast\otimes\Gamma_\ast^{\otimes s}$ and $\sigma^2 b_i$ equals $x\otimes\sigma^2 b_i\in A_\ast\otimes\Gamma_\ast^{\otimes (s+1)}$.
	\end{remark}

	\begin{lemma}\label{lem:dl}
	    For any choice of generators in Definition \ref{def:laz}, we have
	    \[
	        (\sigma^2 v_k)^p \equiv \sigma^2 v_{k+1} \pmod{p,\beta}
	    \]
	    for $k\geq 2$ and
	    \[
	        (\sigma^2 x_{p-1})^p \equiv\sigma^2 v_2 \pmod{p,\beta}
	    \]
	    in $A_\ast$ up to a $p$-adic unit. Similarly, we have
	    \[
	        (\sigma^2 b_i)^p \equiv \sigma^2 b_{pi+p-1} \pmod{p,x_1,x_2,\ldots}
	    \]
	    for $i\geq 1$ in $\Gamma_\ast$ up to a $p$-adic unit.
	\end{lemma}
	\begin{proof}
	    For $A_\ast$, the proof is the same as the proof of \cite[Prop. 2.5.3]{HW}. For $\Gamma_\ast$, it is similar and we shall sketch the proof. We wish to show that $(\sigma^2 v_k)^p=\sigma^2 v_{k+1}$ in
	    \[
	        \Gamma_\ast \otimes_{\MU_\ast} \mathbb F_p = \pi_\ast (\Gamma\otimes_\MU\mathbb F_p).
	    \]
	    Since base changing along $\MU\to\mathbb F_p$ is a symmetric monoidal functor, we have
	    \[
	        \Gamma\otimes_{\MU}\mathbb F_p = \mathbb F_p\otimes_{\THH(\MU;\mathbb F_p)}\mathbb F_p
	    \]
	    where $\THH(\MU;\mathbb F_p) = \mathbb F_p\otimes_{\mathbb F_p\otimes\MU}\mathbb F_p$ again by base change. By the stability of Dyer-Lashof operations, we have
	    \[
	        (\sigma^2 b_i)^p = Q_0(\sigma^2 b_i) = \sigma^2 (Q_2 b_i).
	    \]
	    Then, the statement follows from the computation of the operation $Q_2$ in $(\mathbb F_p)_{\ast}\MU=H_\ast(BU;\mathbb F_p)$, done in \cite[Thm. 6]{Kochman}.
	\end{proof}
	
	\begin{lemma}\label{lem:coaction}
		We can choose the generators $x_2,x_3,\dots,b_1,b_2,\dots$ so that the following properties hold.
		\begin{enumerate}
		    \item There is a sequence of $p$-adic units $\delta_0=1, \delta_1,\delta_2,\dots\in\mathbb Z_p^\times$ such that $(\sigma^2 v_2)^{p^k}\equiv\delta_k\sigma^2 v_{k+2}\pmod p$ for $k\geq 0$ in $A_\ast$.
		    \item We have $(\sigma^2 t_1)^{p^k}\equiv\delta_k\sigma^2 t_{k+1}\pmod p$ for $k\geq0$ in $\Gamma_\ast$ with the same sequence $\delta_1,\delta_2,\dots$ as in (a).
		    \item We have $(\sigma^2 b_1)^p \equiv \sigma^2 b_{2p-1} \pmod p$ in $\Gamma_\ast$.
		    \item The coaction of the element $\sigma^2 v_k\in A_\ast$ is given as
		    \[
		        \eta_R \sigma^2 v_k = \sigma^2 v_k + p\sigma^2 t_k + \beta^{p^k-p^{k-1}}\sigma^2 t_{k-1}
		    \]
		    and the coaction of the element $\sigma^2 x_{p-1}$ is given as
		    \[
		        \eta_R \sigma^2 x_{p-1} = \sigma^2 x_{p-1}+p\sigma^2 t_1 + \beta^{p-2}\sigma^2 b_1.
		    \]
		    \item There is a constant $\delta'$ such that the coaction of the element $\sigma^2 x_{2p-1}$ is given as
		    \[
		        \eta_R\sigma^2 x_{2p-1} = \sigma^2 x_{2p-1} + \sigma^2 b_{2p-1} + \delta' \beta^p\sigma^2 t_1.
		    \]
		    \item The constant $\delta'$ in (e) is a $p$-adic unit.
		\end{enumerate}
	\end{lemma}
	\begin{proof}
        
	We first note that we do not need to require that the sequence $\delta_1,\delta_2,\dots$ of $p$-adic units appearing in (a) and (b) are the same sequences because it follows automatically from the first equation of (d) by taking $p$-th powers mod $p$ and the fact that the statement only depends on the mod $p$ reductions of the $\delta_i$'s.
    
	    By the naturality of $\sigma$, we have $\eta_R(\sigma^2 \alpha) = \sigma^2 \eta_R(\alpha)$ for any $\alpha\in \pi_\ast\operatorname{fib}(\MU\to\ku)$. Therefore, the lemma will be proved using properties of the right unit $\MU_\ast\to\MU_\ast\MU$.

        \textbf{Setup. }Let us outline our strategy. Starting from some set of polynomial generators for $\MU_\ast$ and $\MU_\ast\MU$ as in Definition \ref{def:laz}, we shall try to modify the generators so that they satisfy the lemma.
        
        Let us write $I_0=(x_2,x_3,\dots)\subseteq \MU_\ast$ and $I_1=(x_2,x_3,\dots,b_1,b_2,\dots)\subseteq \MU_\ast\MU$ for the kernel ideals of maps to $\ku_\ast$. Being kernels of maps to $\ku_\ast$ implies that $\eta_L$ and $\eta_R$ map $I_0$ to $I_1$. Also, by Lemma \ref{lem:sus}, $\sigma^2$ annihilates $I_0^2$ and $I_1^2$. In other words, there is a commutative diagram
        \[
            \begin{tikzcd}
                I_0/I_0^2\ar[r,"\overline{\eta_?}"]\ar[d,"\sigma^2"]&I_1/I_1^2\ar[d,"\sigma^2"]\\
                A_\ast\ar[r,"\eta_?"]&A_\ast\otimes\Gamma_\ast
            \end{tikzcd}
        \]
        for $?=L,R$ where $\overline{\eta_?}$ denotes the top arrow.

        The quotients $I_0/I_0^2$ and $I_1/I_1^2$ are naturally modules over $\MU_\ast/I_0=\MU_\ast\MU/I_1=\ku_\ast$. As a $\ku_\ast$-module, $I_0/I_0^2$ is a free graded module generated by $\{\ol{x_2},\ol{x_3},\dots\}$, where $\ol{x_i}$ is the image of $x_i\in I_0$. Similarly, $I_1/I_1^2$ is free and there is a natural splitting
        \[
            I_1/I_1^2=\ol{\eta_L}(I_0/I_0^2) \oplus \ku_\ast\{\ol{b_1},\ol{b_2},\dots\}
        \]
        induced by the multiplication map $\MU_\ast\MU\to\MU$.
        
        We also observe that $\overline{\eta_L}$ and $\ol{\eta_R}$ are $\ku_\ast$-linear. For $\ol{\eta_L}$, this is clear since $\eta_L(x_1)=x_1$. For $\ol{\eta_R}$, this can be seen by either observing that the vertical arrows in the above square are injective or by a direct computation:
        \[
            \eta_R(x_1\alpha)-x_1\eta_R(\alpha)=b_1\eta_R(\alpha)\in I_1^2\qquad\forall\alpha\in I_0.
        \]

        For the purpose of proving this Lemma, we only need to choose a set of generators of $I_0/I_0^2$ and $I_1/I_1^2$ as free $\ku_\ast$-modules. Any set of generators lift to a set of polynomial generators for $\MU_\ast$ and $\MU_\ast\MU$, and the choice of a lift is irrelevant since $\sigma^2$ annihilates any difference.

        \textbf{The reduction argument. }For any $k\geq 2$, we have
        \begin{equation}\label{eq:coeffs}
            \ol{\eta_R}(\ol{x_k}) = \ol{x_k} + \sum_{i=1}^k c_{k,i} \beta^{k-i}\ol{b_i}
        \end{equation}
        for some $c_{k,0},\dots,c_{k,k}\in\mathbb Z_p$. It is a standard fact \cite[Theorem 3.1.5]{green} that the top coefficient $c_{k,k}$ is a $p$-adic unit unless $k+1$ is a power of $p$ and that we have $c_{k,k}\in p\mathbb Z_p^\times$ if $k+1$ is a power of $p$.
        Note that by counitality of the Hopf algebroid $\MU_\ast\MU$, there are no terms of the form $\beta^{k-i}\ol{x_i}$ on the right hand side.

        Let us consider $(c_{k,i})$ as a matrix, indexed by $\{2,3,\dots\}\times\{1,2,\dots\}$, i.e. it is an infinite square matrix with the first row missing. Then, it is ``lower-triangular'' in the sense that $c_{k,i}=0$ for $i>k$, and the ``diagonal'' entries $c_{k,k}$ are units except for those such that $k+1$ is power of $p$.

        Since $\ol{\eta_R}$ is $\ku_\ast$-linear, modifying a generator of $I_0/I_0^2$, by adding multiples of $\beta^{k-j}x_j$ to $x_k$ for example, has an effect of adding the same multiple of the $j$-th row to the $k$-th row. As the diagonal entries are mostly units, we can perform the row operations to make the matrix diagonal except on the columns with a non-unit diagonal entry and the first column. More precisely, suppose we have
        \[
            \ol{\eta_R}(\ol{x_k})=\ol{x_k}+\sum_{i=1}^k c_{k,i}\beta^{k-i}\ol{b_i}.
        \]
        Then, we take the greatest index $i_0\in \{2,\dots,k-1\}\backslash\{p^j-1:j\geq1\}$ such that $c_{k,i_0}\neq 0$, if one exists. Then, if we replace $\ol{x_k}$ by $\ol{x_k}-c_{i_0,i_0}^{-1}c_{k,i_0}\beta^{k-i_0}\ol{x_{i_0}}$, we effectively make $c_{k,i_0}=0$. Repeating this, we arrive at a generator $\ol{x_k}$ such that
        \[
            \ol{\eta_R}(\ol{x_k}) = \ol{x_k} + c_{k,k}\ol{b_k} + \left(\sum_{p^j-1<k}c_{k,p^j-1}\beta^{k-p^j+1}\ol{t_j}\right) + c_{k,1}\beta^{k-1}\ol{b_1}.
        \]
        This \emph{reduction argument} will be used multiple times in the following proof.
	    
	\textbf{Step 1: choosing $x_{p-1},t_1,$ and $b_1$ to satisfy (d). }Let $i_Q:(\BP_\ast,\BP_\ast\BP)\to(\MU_\ast,\MU_\ast\MU)$ be the map of Hopf algebroids induced by the Quillen idempotent and let $v_1',v_2',\dots,t_1',t_2',..$ be the Hazewinkel generators (\cite[A2.2]{green}) for $\BP_\ast=\mathbb Z_p[v_1',\dots]$ and $\BP_\ast\BP=\BP_\ast[t_1',\dots]$. Note that $i_Q(v_1')\in\MU_\ast$ is the coefficient of $X^p$ in the $p$-series of the universal formal group law mod $p$. This means that the image of $i_Q(v_1')$ in $\ku_\ast$ is $\varepsilon\beta^{p-1}$ for some $p$-adic unit $\varepsilon$. We define $x_{p-1}\in\MU_\ast$ to be
	    \[
	    x_{p-1}:=i_Q(v_1') - \varepsilon x_1^{p-1}\in I_0
	    \]
	    and define $t_1\in\MU_\ast\MU$ to be $i_Q(t_1')$. If we choose $b_1$ so that $\eta_R(x_1)=x_1+b_1$ at least for now, then we have
	    \begin{align*}
	        \eta_R(x_{p-1}) &= i_Q(\eta_R v_1) - \varepsilon \eta_R(x_1^{p-1}) \\
	        &= i_Q(v_1'+pt_1') - \varepsilon(x_1+b_1)^{p-1} \\
	        &= x_{p-1}+pt_1 - \varepsilon x_1^{p-2}b_1 \pmod{I_1^2}.
	    \end{align*}
	    Therefore, after scaling $b_1$ by a $p$-adic unit, we have
        \[
            \ol{\eta_R}(\ol{x_{p-1}})=\ol{x_{p-1}}+p\ol{t_1}+\beta^{p-2}\ol{b_1}
        \]
        inducing the second equation of (d) after taking $\sigma^2$. After this scaling, we also have $\eta_R(x_1)=x_1-\varepsilon^{-1}b_1$.
	    
	    \textbf{Step 2: choosing $b_i$'s for $i\neq1,p-1$ to satisfy (b) and (c). }Make arbitrary choices for all $\ol{b_i}$'s that have not already been defined. We shall redefine $\ol{t_k}$'s inductively in $k(\geq 2)$ so that (b) holds. Suppose that we have chosen $\ol{t_{k-1}}$. By Proposition \ref{prop:hopf}, we have
	    \[
	        (\sigma^2 t_{k-1})^p \equiv \alpha_0 \sigma^2 t_k + \sum_{i=1}^{p^k-2} \alpha_i\beta^i \sigma^2 b_{p^k-1-i}\pmod p
	    \]
	    for some constants $\alpha_i$'s. By Lemma \ref{lem:dl}, $\alpha_0$ is a $p$-adic unit. Therefore, we may redefine $\ol{t_k}$ to be
	    \[
	        \alpha_0\ol{t_k}+\sum_{i=1}^{p^k-2}\alpha_i\beta^i \ol{b_{p^k-1-i}}
	    \]
	    then $(\sigma^2 t_{k-1})^p$ would be a unit multiple of $\sigma^2 t_k$ modulo $p$.
	    
	    By the same argument, we can redefine $\ol{b_{2p-1}}$ to satisfy (c).
	    
	    \textbf{Step 3: choosing $v_2$ to satisfy (d). }We shall now choose $\ol{v_2}$. Recall that $v_2'\in\BP_\ast\subseteq\MU_\ast$ is the Hazewinkel generator. Suppose that the image of $v_2'$ in $\ku_\ast$ is $c\beta^{p^2-1}$ for some $c\in\mathbb Z_p$. Let us for now define
        \[
            v_2= i_Q(v_2')-cx_1^{p^2-1}.
        \]
        Using the definition of Hazewinkel generators and formulas in \cite[A.2.1.27]{green}, it can be computed that
        \[
            \eta_R(v_2')=v_2'+pt_2'-(p+1)(v_1')^{p+1}t_1' \pmod{(t_1')^2}
        \]
        in $\BP_\ast\BP$. Therefore,
        \begin{align*}
            &\eta_R(v_2) \\
            &= i_Q(v_2')+pi_Q(t_2')-(p+1)i_Q(v_1')^{p+1}i_Q(t_1') - c(x_1-\varepsilon^{-1} b_1)^{p^2-1} \pmod{I_1^2}\\
            &= v_2 + pi_Q(t_2')-(p+1)\varepsilon^{p+1}x_1^{p^2-p}i_Q(t_1') + (p^2-1)c\varepsilon^{-1} x_1^{p^2-2}b_1 \pmod{I_1^2}
        \end{align*}
        where we have used that $i_Q(v_1')=\varepsilon x_1^{p-1}\pmod{I_1}$ from Step 1. After writing $i_Q(t_1')$ and $i_Q(t_2')$ in terms of our generators $b_i$'s from Steps 1 and 2, we conclude that in the equation \eqref{eq:coeffs}
        \[
            \ol{\eta_R}(\ol{v_2}) = \ol{v_2} + \sum_{i=1}^{p^2-1} c_{p^2-1,i}\beta^{p^2-1-i}\ol{b_i},
        \]
        we have that the coefficients $c_{p^2-1,i}$ are multiples of $p$ for $p-1<i\leq p^2-1$ and is a $p$-adic unit for $i=p-1$.

        Then, using the reduction argument, we can modify $\ol{v_2}$ so that its $\ol{\eta_R}$ is of the form
        \[
            \ol{\eta_R}(\ol{v_2})=\ol{v_2}+c_{p^2-1,p^2-1}\ol{t_2}+c_{p^2-1,p-1}\beta^{p^2-p}\ol{t_1}+\beta^{p^2-2}c_{p^2-1,1}\ol{b_1}
        \]
        for some $c_{p^2-1,p^2-1}\in p\mathbb Z_p^\times$. Since the previous coefficients $c_{p^2-1,i}$ were multiples of $p$ for $p-1<i<p^2-1$, this implies that the coefficient $c_{p^2-1,p-1}$ is still a $p$-adic unit after the reduction argument. Then, by redefining $\ol{v_2}$ to be $\ol{v_2}-c_{p^2-1,1}\beta^{p^2-p}\ol{x_{p-1}}$, we can also make $c_{p^2-1,1}=0$ while retaining that $c_{p^2-1,p-1}\in\mathbb Z_p^\times$. Finally, by scaling $\ol{v_2}$ and $\ol{t_2}$ if necessary, we obtain the desired formula
        \[
            \ol{\eta_R}(\ol{v_2})=\ol{v_2}+p\ol{t_2} + \beta^{p^2-p}\ol{t_1}.
        \]

        \textbf{Step 4: choosing $x_{2p-1}$ to satisfy (e).} The argument is similar to the one in Step 3. Starting from an arbitrary choice of $\ol{x_{2p-1}}$, we can modify it with the reduction argument to obtain
        \[
            \eta_R(\ol{x_{2p-1}}) = \ol{x_{2p-1}}+c_{2p-1,2p-1}\ol{b_{2p-1}} + c_{2p-1,p-1}\beta^p\ol{t_1}+c_{2p-1,1}\beta^{2p-2}\ol{b_1}.
        \]
        By adding a multiple of $\beta^p\ol{x_{p-1}}$ to $\ol{x_{2p-1}}$, we may also assume that $c_{2p-1,1}=0$. Finally, by scaling $x_{2p-1}$, we may assume that $c_{2p-1,2p-1}=1$ and set $\delta':=c_{2p-1,p-1}$.
        
        The difference from the previous step is that we cannot scale $\delta'$ to a fixed number because $b_{2p-1}$ cannot be scaled considering the consistency with (c).
	    
	\textbf{Step 5: choosing $v_k$'s for $k\geq3$.} By the same argument as in Step 2, we can choose $v_3,v_4,\dots$ satisfying (a).
	    
	    Then, we shall modify $v_k$'s inductively for $k\geq 3$ so that (d) is true. Suppose that we have chosen $v_{k-1}$ so that (d) is true. Consider the following two equations
	    \begin{align*}
	        \eta_R \sigma^2 v_{k-1} &= \sigma^2 v_{k-1} + p\sigma^2 t_{k-1} + \beta^{p^{k-1}-p^{k-2}}\sigma^2 t_{k-2}\\
            \eta_R\sigma^2 v_k &=\sigma^2v_k + \sum_{i=1}^{p^k-1}c_{p^k-1,i}\beta^{p^k-1-i}\sigma^2 b_i.
	    \end{align*}
            Taking the $p$-th power mod $p$ of the first and comparing with the second using (a) and (b), we see that $c_{p^k-1,i}$'s are multiples of $p$ for $i\neq p^{k-1}-1$ and that $c_{p^k-1,p^{k-1}-1}$ is a unit.

            Then, as in Step 3, we can perform the reduction argument to obatin
            \[
                \ol{\eta_R}(\ol{v_k})=\ol{v_k} + \sum_{j=1}^k c_{p^k-1,p^j-1} \beta^{p^k-p^j} \ol{t_j} + c_{p^k-1,1}\beta^{p^k-2}\ol{b_1}
            \]
            for some $c_{p^k-1,p^k-1}\in p\mathbb Z_p^\times$. Since we the $c_{p^k-1,i}$'s were previously multiples of $p$ for $i\neq p^{k-1}$, the reduction argument does not change $\ol{v_k}$ mod $p$ so that (a) still holds. Also, it implies that the reduction argument does not change the coefficients $c_{p^k-1,i}$ mod $p$, so that $c_{p^k-1,p^{j}-1}$ is a unit for $j=k-1$ and is a multiple of $p$ if $j<k-1$, and also that $c_{p^k-1,1}$ is a multiple of $p$.

            By adding a multiple of $\beta^{p^k-p}\ol{x_{p-1}}$ to $\ol{v_k}$, we can make $c_{p^k-1,1}=0$. This doesn't change anything mod $p$ since $c_{p^k-1,1}$ was previously a multiple of $p$. Then, inductively in $j\in\{2,\dots,k-1\}$, we can add a multiple of $\beta^{p^k-p^j}\ol{v_j}$ to make $c_{p^k-1,p^{j-1}-1}=0$. In the end, we are left with
            \[
                \ol{\eta_R}(\ol{v_k}) = \ol{v_k} + c_{p^k-1,p^k-1}\ol{t_k}+c_{p^k-1,p^{k-1}-1}\beta^{p^k-p^{k-1}}\ol{t_{k-1}}
            \]
            where $c_{p^k-1,p^k-1}\in p\mathbb Z_p^\times$ and $c_{p^k-1,p^{k-1}-1}\in\mathbb Z_p^\times$. After scaling $\ol{v_k}$ and $\ol{t_k}$, we obtain
            \[
                \ol{\eta_R}(\ol{v_k})=\ol{v_k}+p\ol{t_k}+\beta^{p^k-p^{k-1}}\ol{t_{k-1}}.
            \]

	    \textbf{Step 6: proof of (f). }By Proposition \ref{prop:hopf} and Lemma \ref{lem:dl}, we have
	    \[
	        (\sigma^2 x_{p-1})^p \equiv \alpha_0 \sigma^2 v_2 + \sum_{i=1}^{p^2-2} \alpha_i \beta^i \sigma^2 x_{p^2-1-i} \pmod p
	    \]
	    for some constants $\alpha_0,\dots,\alpha_{p^2-2}$ with $\alpha_0$ a unit. We shall compare both sides after applying $\eta_R$. Since
	    \[
	        \eta_R((\sigma^2 x_{p-1})^p) \equiv (\sigma^2 x_{p-1})^p + \beta^{p(p-2)}\sigma^2 b_{2p-1}\pmod p
	    \]
	    we have $\alpha_i\equiv0\pmod p$ for $i\neq 0,p^2-2p,p^2-p$ since it would otherwise introduce a nonzero term $\beta^i \sigma^2 b_{p^2-1-i}$ when $\eta_R$ is applied. If $\delta'\equiv0\pmod p$, then on the right hand side, we would not be able to cancel out $\beta^{p^2-p}\sigma^2 t_1$ appearing in $\eta_R(\alpha_0 \sigma^2 v_2)$.
	\end{proof}
	
	\begin{remark}
	    Using that $\sigma^2 b_i$ can be identified with the image of Bott map in \cite[Prop. II.12.6]{blue}, it could be possible to determine the sequence $(\delta_i)$ or even the generators.
	\end{remark}
	
	\section{Bockstein Spectral Sequences}
	
	As before, until the end of this section, $p$ is a fixed odd prime and every spectrum is $p$-complete. We also fix a set of generators $x_1,x_2,\dots,b_1,b_2,\dots$ of $\MU_\ast$ and $\MU_\ast\MU$ satisfying Lemma \ref{lem:coaction}.
	
	Recall that $\CB_{\Gamma_\ast}(A_\ast)$ is a cochain complex of free $\ku_\ast$-modules. Then, filtering this cobar complex by powers of $\beta$, we obtain the \emph{$\beta$-Bockstein spectral sequence}
	\begin{equation}\label{eq:betabss}
		E_1 = \Ext_{\Gamma_\ast}(A_\ast/\beta)[\beta] \Rightarrow \Ext_{\Gamma_\ast}(A_\ast),
	\end{equation}
	and filtering $\CB_{\Gamma_\ast}(A_\ast/\beta)$ by powers of $p$, we obtain the \emph{$v_0$-Bockstein spectral sequence}
	\begin{equation}\label{eq:v0bss}
	E_1 = \Ext_{\Gamma_\ast}(A_\ast/(p,\beta))[v_0] \Rightarrow \Ext_{\Gamma_\ast}(A_\ast/\beta).
	\end{equation}
	The convergence of the $v_0$-Bockstein spectral sequence follows from the $p$-completeness and the convergence of the $\beta$-Bockstein spectral sequence follows from the fact that $A_\ast\otimes\Gamma_\ast^{\otimes s}$ is bounded below for any $s$. We shall compute these spectral sequences in this section.
	
	Note that a $v_0$-Bockstein differential $d_r(x)=v_0^ry$ is equivalent to finding a class $\widetilde x\in\CB_{\Gamma_\ast}^i(A_\ast/\beta)$ such that $\widetilde x$ represents $x$ modulo $p$ and $D^i(\widetilde x)=p^r\widetilde y$ for some $\widetilde y\in \CB_{\Gamma_\ast}^{i+1}(A_\ast/\beta)$ representing $y$ modulo $p$. There is a similar description for $\beta$-Bockstein differentials.
	
	\subsection{\texorpdfstring{$v_0$}{v0}-Bockstein}\label{ssec:v0b}
	\begin{proposition}\label{prop:mainp}
		We have
		\[
			\Ext_{\Gamma_\ast}(A_\ast/(p,\beta)) = \mathbb F_p[\mu]\otimes_{\mathbb F_p} \Lambda_{\mathbb F_p}(\lambda_1,\lambda_2)
		\]
		as $\mathbb F_p$-algebras, where $\mu,\lambda_1,\lambda_2$ are represented by $\sigma^2 x_{p-1}, \sigma^2 t_1,\sigma^2 b_1$, respectively, in $\CB_{\Gamma_\ast}(A_\ast/(p,\beta))$.
	\end{proposition}
	\begin{proof}
	    We can mimic the proof of \cite[Prop. 6.1.6]{HW}. The necessary ingredient is \cite[Thm. 6.8]{A-THHku} which states that
	    \[
	        \THH_\ast(\ku;\mathbb F_p) \simeq \mathbb F_p[\sigma^2 x_{p-1}] \otimes \Lambda_{\mathbb F_p}(\sigma t_1,\sigma b_1).
	    \]
	\end{proof}

	\begin{corollary}\label{cor:deg}
	    The descent spectral sequence for 
	    \[
	    \THH(\ku)\to\THH(\ku/\MU)
	    \]
	    degenerates at the $E_2$-page.
	\end{corollary}
	\begin{proof}
	    Since $\Ext^{s,t}_{\Gamma_\ast}(A_\ast/(p,\beta))$ is concentrated in $0\leq s\leq2$, so is $\Ext^{s,t}_{\Gamma_\ast}(A_\ast)$. This follows from the convergence of Bockstein spectral sequences and will become clearer as we compute Bockstein differentials. Furthermore, the Ext groups are nonzero only if $t\in 2\mathbb Z$ because this is already true for cobar complexes. Therefore, there is no room for differentials and the descent spectral sequence degenerates.
	\end{proof}
	
	\begin{lemma}\phantomsection\label{lem:luc}
	    \begin{enumerate}
	    \item For any nonnegative integers $m,k$, we have
	    \[
	        \val_p\binom mk\geq \val_p(m)-\val_p(k),
	    \]
	    and if $k$ is a power of $p$ and $m$ is a multiple of $k$, then it is an equality.
	    \item Recall that $p>2$. Let $R$ be any commutative ring and $x,y\in R$ be any elements. Then
	    \[
	        (x+p^e y)^m \equiv x^m + mp^e x^{m-1}y \pmod{p^{e+\val_p(m)+1}} 
	    \]
	    for any positive integers $m,e$. 
	    \item If $m$ is a multiple of $p^k$ for some $k\geq 0$, then we have
	    \[
	        \binom{m}{p^k} \equiv \frac{m}{p^k} \pmod p.
	    \]
	    \end{enumerate}
	\end{lemma}
	\begin{proof}
	    (a) follows from the fact that $\val_p\binom mk$ is equal to the number of carries in the $p$-adic addition of $k$ and $m-k$. (b) follows from (a) since the $k$'th term in the binomial expansion has $p$-adic valuation
        \[
            \val_p\left[\binom mk p^{ke}\right] \geq \val_p(m)-\val_p(k) + ke\geq \val_p(m)+e-\log_p k+(k-1)e\geq\val_p(m)+e+1
        \]
        for $k\geq 3$. The last inequality holds since
        \[
            -\log_pk+(k-1)e\geq-\log_pk+(k-1)\geq 1
        \]
        is implied by the inequalities $p\geq3$ and $3^{k-2}\geq k$. If $k=2$, then $\val_p(k)=0$, so that 
        \[
            \val_p\left[\binom mk p^{ke}\right] \geq \val_p(m)+ke\geq\val_p(m)+e+1.
        \]
        Note that we are using our assumption that $p$ is odd.
        (c) is a special case of Lucas's theorem.
	\end{proof}

        \begin{remark}
            Lemma \ref{lem:luc} (b) is false if $p=2$, $e=1$, and $m\equiv2\pmod4$, which is not a problem for us since we are assuming that $p$ is odd. One needs to be slightly more careful if one wants to replicate the rest of the argument in the paper at $p=2$.
        \end{remark}
	
	\begin{proposition}
		We have
		\[
			\Ext_{\Gamma_\ast}(A_\ast/\beta) \simeq \left[\mathbb Z_p\{1\}\oplus\bigoplus_{k=1}^\infty \mathbb Z/p^{\val_p(k)+1}\{a_k\}\right]\otimes_{\mathbb Z_p}\Lambda_{\mathbb Z_p}(\lambda_2)
		\]
		as $\mathbb Z_p$-modules. The generators are represented in the $\CB_{\Gamma_\ast}(A_\ast/\beta)$ as the following:
		\begin{itemize}
			\item $\mathbb Z_p$ is generated by $1$,
			\item The generator $a_k$ of $\mathbb Z/p^{\val_p(k)+1}$ is represented by the bidegree $(1,2pk)$ cycle (also denoted $a_k$ by abuse of notation)
			\[
				a_k:=\frac{D^0((\sigma^2 x_{p-1})^k)}{p^{\val_p(k)+1}} = \frac{(\sigma^2 x_{p-1}+p\sigma^2 b_{p-1})^k-(\sigma^2 x_{p-1})^k}{p^{\val_p(k)+1}},
			\]
			\item $\lambda_2$ is represented by $\sigma^2 b_1$.
		\end{itemize}
		Multiplication by $\lambda_2$ should be interpreted as in Remark \ref{rmk:mul}.
	\end{proposition}
	\begin{proof}
	    We work in $\CB_{\Gamma_\ast}(A_\ast/\beta)$. Then, by Lemma \ref{lem:luc}(b), we have
	    \begin{align*}
	        D^0((\sigma^2 x_{p-1})^k) &= (\sigma^2 x_{p-1}+p\sigma^2 t_1)^k - (\sigma^2 x_{p-1})^k \\
	        &\equiv kp(\sigma^2 x_{p-1})^{k-1}\sigma^2 t_1 \pmod{p^{\val_p(k)+2}},
	    \end{align*}
	    which gives us the differentials
	    \[
	        d_{\val_p(k)+1}(\mu^k) = v_0^{\val_p(k)+1}\mu^{k-1}\lambda_1
	    \]
	    up to a $p$-adic unit in the $v_0$-Bockstein spectral sequence \eqref{eq:v0bss}. Since $\lambda_2$ is a permanent cycle represented by $\sigma^2 b_1$, we have the differentials
	    \[
	        d_{\val_p(k)+1}(\mu^k \lambda_2) = v_0^{\val_p(k)+1}\mu^{k-1}\lambda_1\lambda_2.
	    \]
	    The generators in the statement can be derived from this computation.
	\end{proof}
	
	\begin{corollary}
	    We have
	    \[
	        \THH_\ast(\ku;\mathbb Z_p) = \left[\mathbb Z_p\oplus\bigoplus_{k=1}^\infty \Sigma^{2pk-1}\mathbb Z/p^{\val_p(k)+1}\right]\otimes_{\mathbb Z_p}\Lambda_{\mathbb Z_p}(\lambda_2)
	    \]
	    as graded $\mathbb Z_p$-modules with $|\lambda_2|=3$.
	\end{corollary}
	\begin{proof}
	    Since
	    \[
	        E_2^{s,t}(\THH(\ku;\mathbb Z_p)) = \Ext_{\Gamma_\ast}^{s,t}(A_\ast/\beta)
	    \]
	    is concentrated in $0\leq s\leq 2$ and $t\in 2\mathbb Z$, there is no room for any further differential in the descent spectral sequence computing $\THH(\ku;\mathbb Z_p)$. There is no extension problem since $E_2^{0,\ast}$ is free as a $\mathbb Z_p$-module.
	\end{proof}
	
	\begin{remark}
	    The previous corollary agrees with the known result from \cite[Cor. 6.9]{A-THHku}.
	\end{remark}
	
	\subsection{\texorpdfstring{$\beta$}{beta}-Bockstein}
	The differentials of the $\beta$-Bockstein spectral sequence are given as the following.
	
	\begin{theorem}\phantomsection\label{thm:bbstn}
		\begin{enumerate}
			\item For all $e\geq0$, $p^e a_{p^e}\in\Ext^1_{\Gamma_\ast}(A_\ast/\beta)$ can be lifted to a class in $\Ext^1_{\Gamma_\ast}(A_\ast)$ represented by the cycle $\sigma^2 t_{e+1}\in\CB^1_{\Gamma_\ast}(A_\ast)$.
			\item For each $0\leq e\leq \val_p(m)$, there is a differential
			\[
				d_{p^{e+1}-2}(p^ea_m) = \frac{m-p^e}{p^e} \beta^{p^{e+1}-2}a_{m-p^e}\lambda_2
			\]
			up to a $p$-adic unit in the $\beta$-Bockstein spectral sequence \eqref{eq:betabss}.
		\end{enumerate}
	\end{theorem}
	
	To prove the theorem, let us first construct some elements in the cobar complex. Define
	\[
	    y = (\delta')^{-1}(-(\sigma^2 x_{p-1})^p+\beta^{p(p-2)}\sigma^2 x_{2p-1})\in A_\ast.
	\]
	
	\begin{lemma}\phantomsection\label{lem:ydef}
	    \begin{enumerate}
	        \item The coaction modulo $p^2$ of $y$ is given as
	        \[
	            \eta_R(y) \equiv y + \beta^{p(p-1)}\sigma^2t_1\pmod{p^2}.
	        \]
	        \item We have $y\equiv\sigma^2 v_2\pmod p$.
	    \end{enumerate}
	\end{lemma}
	\begin{proof}
	    (a) follows from Lemma \ref{lem:coaction} and Lemma \ref{lem:luc}(b).
	    
	    By Proposition \ref{prop:hopf}, $y$ is a $\ku_\ast$-linear combination of $\sigma^2 x_i$'s. However, using Lemma \ref{lem:coaction}, we can see that $\eta_R(y) \pmod p$ completely determines $y\pmod p$ if $y$ is a linear combination of $\sigma^2 x_i$'s.
	\end{proof}
	
	Next, fix a pair of positive integers $(m,e)$ such that $1\leq e\leq \val_p(m)$. Then, we define constants $\varepsilon_1,\dots,\varepsilon_e$ inductively as
	\begin{align}
	    \varepsilon_1 = -\frac{\delta_{e-1}}{p^{\val_p(m)-e}}\binom{m/p}{p^{e-1}},\qquad
	    \varepsilon_{k+1} = -\varepsilon_k \left(1+\binom{(m-p^e)/p}{p^{e-k-1}}\right).\label{eq:eps}
	\end{align}
	Then, define
	\begin{align*}
	    f_{m,e} := &\frac1{p^{\val_p(m)-e+1}}D^0(y^{m/p}) \\
	    &+ \sum_{k=1}^{e} \frac{\varepsilon_k}{p^{k+1}}\beta^{p^{e+1}-p^{e-k+1}} D^0\left(y^{(m-p^e)/p}\sigma^2 x_{p^{e-k+1}-1}\right)
	\end{align*}
	which is an element in $(A_\ast\otimes\Gamma_\ast)[p^{-1}]$. We also define
	\[
	    f_{m,0} := \frac1{p^{\val_p(m)+1}} D^0((\sigma^2 x_{p-1})^m).
	\]
	\begin{lemma}\phantomsection\label{lem:epsp}
	    \begin{enumerate}
	        \item $\varepsilon_1,\dots,\varepsilon_e$ are $p$-adic units.
	        \item For $1\leq k<e$, we have $\varepsilon_{k+1}\equiv-\varepsilon_k\pmod{p^k}$. If $\val_p(m-p^e)>e$, then we have $\varepsilon_{k+1}\equiv-\varepsilon_k\pmod{p^{k+1}}$.
	    \end{enumerate}
	\end{lemma}
	\begin{proof}
	    Both statements follow from Lemma \ref{lem:luc}.
	\end{proof}
	\begin{lemma}\label{lem:ratc}
	    The image of $f_{m,e}$ in $(A_\ast\otimes\Gamma_\ast)/(p^\infty, \beta^{p^{e+1}-1})$ is
	    \[
	        \frac1{p^{e+1}}\beta^{p^{e+1}-2}(\sigma^2 x_{p-1})^{m-p^e}\sigma^2 b_1.
	    \]
	    up to a $p$-adic unit. Here, we write $M/p^\infty:=M[p^{-1}]/M$ for an abelian group $M$.
	\end{lemma}
	\begin{proof}
	    For $e=0$, we have
	    \begin{align*}
	    &D^0((\sigma^2 x_{p-1})^m) \\&\equiv (\sigma^2 x_{p-1} + \beta^{p-2}\sigma^2b_1)^m - (\sigma^2 x_{p-1})^m &\pmod{p^{\val_p(m)+1}}\\
	    &\equiv m\beta^{p-2}(\sigma^2 x_{p-1})^{m-1}\sigma^2 b_1 &\pmod{\beta^{p-1},p^{\val_p(m)+1}}
	    \end{align*}
	    by Lemma \ref{lem:luc} and Lemma \ref{lem:coaction}, and the statement follows.
	    
	    For $e>0$, we similarly have
	    \begin{align*}
	        D^0(y^{m/p}) &\equiv (y + \beta^{p(p-1)}\sigma^2 t_1)^{m/p} - y^{m/p} &\pmod{p^{\val_p(m)+1}}\\
	        &\equiv \binom{m/p}{p^{e-1}} \beta^{p^e(p-1)}y^{(m-p^e)/p}\sigma^2 t_1^{p^{e-1}}&(\bmod{\beta^{p^{e+1}-1},p^{\val_p(m)-e+1}})\\
	        &\equiv \delta_{e-1}\binom{m/p}{p^{e-1}} \beta^{p^e(p-1)}y^{(m-p^e)/p}\sigma^2 t_e&(\bmod{\beta^{p^{e+1}-1},p^{\val_p(m)-e+1}}).
	    \end{align*}
	    Also, using Lemma \ref{lem:luc} we can check that
	    \begin{align*}
	        &D^0\left(y^{(m-p^e)/p}\sigma^2 v_{e-k+1}\right) \\
	        &\equiv py^{(m-p^e)/p}\sigma^2 t_{e-k+1}
	        \\
	        &\hphantom{==}+ \left(1+\binom{(m-p^e)/p}{p^{e-k-1}}\right)\beta^{p^{e-k}(p-1)} y^{(m-p^e)/p}\sigma^2 t_{e-k}(\bmod{p^{k+1},\beta^{p^{e-k+1}-1}})
	   \end{align*}
	   and similarly
	   \begin{align*}
	        &D^0\left(y^{(m-p^e)/p}\sigma^2 x_{p-1}\right)\\	    
	        &=py^{(m-p^e)/p}\sigma^2 t_1+\beta^{p-2}y^{(m-p^e)/p}\sigma^2 b_1 \pmod{p^{e+1},\beta^{p-1}}.
	    \end{align*}
	    Combining these calculations, we obtain the result.
	\end{proof}
	
	\begin{proof}[Proof of Theorem \ref{thm:bbstn}]
	    (a) If $e=0$,
	    \[
	        \frac{D^0(\sigma^2 x_{p-1})-\beta^{p-2}\sigma^2 b_1}{p}=\sigma^2t_1.
	    \]
	    is a cycle representing $a_1$ modulo $\beta$, and if $e>0$,
	    \[
	        \frac{D^0(y^{p^{e-1}})-\delta_{e-1}\beta^{p^{e+1}-p^e}\sigma^2 t_e}p \equiv D^0\left(\frac{y^{p^{e-1}}-\delta_{e-1}\sigma^2v_{e+1}}{p}\right) + \delta_{e-1}\sigma^2 t_{e+1}.
	    \]
	    is a cycle representing $p^ea_{p^e}$ modulo $\beta$ which is homologous to $\delta_{e-1}\sigma^2 t_{e+1}$. Using Lemmas \ref{lem:coaction} and \ref{lem:ydef}, we can check that the fractions in the equation above are actually elements of $A_\ast$ or $A_\ast\otimes\Gamma_\ast$.
	   
	    (b) By Lemma \ref{lem:ratc}, we have $f^+_{m,e}\in A_\ast\otimes\Gamma_\ast$ and $g_{m,e}\in (A_\ast\otimes\Gamma_\ast)[p^{-1}]$ such that
	    \[
	        f_{m,e} = f_{m,e}^+ - \beta^{p^{e+1}-2} g_{m,e},
	    \]
	    and
	    \[
	        g_{m,e}= \frac1{p^{e+1}}(\sigma^2 x_{p-1})^{m-p^e}\sigma^2 b_1 \pmod \beta
	    \]
	    up to a $p$-adic unit, so that $D^1(g_{m,e})$ represents 
	    \[
	    \frac{m-p^e}{p^e}a_{m-p^e}\lambda_2
	    \]
	    modulo $\beta$.
	    
	    Also, we have
	    \[
	        f^+_{m,e}\equiv f_{m,e}\equiv a_m \pmod\beta,
	    \]
	    and since $D^1(f_{m,e})=0$ by construction, we have
	    \[
	        D^1(f^+_{m,e}) = \beta^{p^{e+1}-2}D^1(g_{m,e})
	    \]
	    which implies the Bockstein differentials.
	\end{proof}
	
	Using Theorem \ref{thm:bbstn}, we can compute $\gr_\beta\Ext_{\Gamma_\ast}(A_\ast)$, the associated graded group  with respect to the $\beta$-adic filtration. We shall discuss these groups and extension problems in \ref{ssc:betaep}, but we have a rough description as follows.
	\begin{itemize}
	    \item $\gr_\beta\Ext^0$ is $\ku_\ast$ and is generated by an element represented by the cycle $1$.
	    \item $\gr_\beta\Ext^1$ is generated by the elements detected by $p^ea_{p^e}$ for $e\geq 0$ and $\lambda_2$. These generators are $p$-torsions and support infinite $\beta$-tower.
	    \item $\gr_\beta\Ext^2$ is generated by elements detected by $a_m\lambda_2$ in $\Ext^1_{\Gamma_\ast}(A_\ast)$, which is $p^j\beta^{p^{\val_p(m)-j+1}-2}$-torsion for each $0\leq j\leq \val_p(m)$ and $p^{\val_p(m)+1}$-torsion.
	\end{itemize}
	
	\subsection{Extension Problems}\label{ssc:betaep}
	We shall resolve the extension problems with respect to the $\beta$-adic filtration thereby proving Theorem \ref{thm:main}.
	
	\begin{proof}[Proof of Theorem \ref{thm:mainp}]
	By Theorem \ref{thm:bbstn}, $\gr_\beta\Ext^1_{\Gamma_\ast}(A_\ast)$ is a free $\mathbb Z/p[\beta]$-module generated by classes represented by the cycles $\sigma^2b_1,\sigma^2t_1,\sigma^2t_2,\dots$. From the following differentials in $\CB_{\Gamma_\ast}(A_\ast)$
	\begin{align*}
        D^0(\sigma^2 v_{e+1}) = p\sigma^2 t_{e+1}+\beta^{p^{e+1}-p^e}\sigma^2 t_e\\
        D^0(\sigma^2 x_{p-1}) = p\sigma^2 t_1 + \delta'\beta^{p-2}\sigma^2 b_1,
	\end{align*}
	we can see that $\Ext^1_{\Gamma_\ast}(A_\ast)$ is isomorphic to $F_p^\wedge$, defined in Section \ref{ssec:main}, by identifying the class in $\Ext^1$ represented by $\sigma^2 t_e$ with $(\beta^{p^e-2}/f(p^e-2))\in F_p^\wedge$ up to a $p$-adic unit. 
	
	Next, let us describe the extension problems in $\Ext^2_{\Gamma_\ast}(A_\ast)$. From Theorem \ref{thm:bbstn}, we see that the class $p^ja_m\lambda_2\in\Ext^1_{\Gamma_\ast}(A_\ast/\beta)$ is hit by the Bockstein differential
	\[
	d_{p^{\val_p(m)-j+1}-2}\left(p^{\val_p(m)-j}a_{m+p^{\val_p(m)-j}}\right).
	\]
	From the proof of that theorem, we see that the class $p^ja_m\lambda_2$ can be lifted to a $\beta^{p^{\val_p(m)-j+1}-2}$-torsion class $h_{m,j}\in\Ext^1_{\Gamma_\ast}(A_\ast)$ represented by the cycle
	\[
	    D^1\left( g_{m+p^{\val_p(m)-j},\val_p(m)-j}\right).
	\]
	
	    Let $d=\val_p(m)$ and $j$ be an integer such that $0\leq j<d$. Then, let $\varepsilon_1,\dots,\varepsilon_{d-j}$ be the sequence defined in \eqref{eq:eps} for the pair $(m+p^{d-j},d-j)$. If in addition $j<d-1$, then let $\varepsilon_1',\dots,\varepsilon_{d-j-1}'$ be the sequence for the pair $(m+p^{d-j-1},d-j-1)$.
	    
	    Case 1: If $j<d-1$, we have
	    \begin{align*}
	        &pf_{m+p^{d-j},d-j} - \frac{\varepsilon_{d-j}}{\varepsilon_{d-j-1}'} \beta^{p^{d-j}(p-1)}f_{m+p^{d-j-1},d-j-1}\\
	        &=\frac{1}{p^{\val_p(m+p^{d-j})-d+j}}D^0(y^{(m+p^{d-j})/p}) \\
	        &\hphantom{=}+ \frac1p\beta^{p^{d-j}(p-1)}\left(\varepsilon_1D^0(y^{m/p}\sigma^2v_{d-j})-\frac{\varepsilon_{d-j}}{\varepsilon_{d-j-1}'}D^0(y^{(m+p^{d-j-1})/p})\right)\\
	        &\hphantom{=}+\sum_{k=1}^{d-j-1}\frac{1}{p^{k+1}}\beta^{p^{d-j+1}-p^{d-j-k}}\left(\varepsilon_{k+1} - \frac{\varepsilon_{d-j}}{\varepsilon_{d-j-1}'}\varepsilon_k'\right)D^0(y^{m/p}\sigma^2 x_{p^{d-j-k}-1}).
	    \end{align*}
	    We shall show that this is integral, i.e. an element of $A_\ast\otimes\Gamma_\ast$, possibly except for the first term. For the second term, we have
	    \[
	        D(y^{(m+p^{d-j-1})/p}) \equiv \delta_{d-j-2} D(y^{m/p}\sigma^2 v_{d-j})\pmod p,
	    \]
	    so that we need to show
	    \[
	        \varepsilon_1 - \frac{\varepsilon_{d-j}}{\varepsilon_{d-j-1}'}\delta_{d-j-2} \equiv 0\pmod p.
	    \]
	    This is true since we have
	    \[
	        \varepsilon_1 - \frac{\varepsilon_{d-j}}{\varepsilon_{d-j-1}'}\delta_{d-j-2} \equiv\varepsilon_1 + \frac{\varepsilon_1}{\varepsilon_{1}'}\delta_{d-j-2}\pmod p
	    \]
	    by Lemma \ref{lem:epsp} and
	    \[
	        \varepsilon_1' + \delta_{d-j-2} = \delta_{d-j-2}\left(-\frac{m+p^{d-j-1}}{p^{d-j-1}}+1\right) \equiv 0\pmod p
	    \]
	    by Lemma \ref{lem:luc}. Next, for the summation part, we wish to show that
	    \[
	        \varepsilon_{k+1} \equiv \frac{\varepsilon_{d-j}}{\varepsilon_{d-j-1}'}\varepsilon_k'\pmod{p^{k+1}},
	    \]
	    and this follows from Lemma \ref{lem:epsp}. Therefore, we have proved that
	    \begin{align}\label{eq:extb}
	        &pf_{m+p^{d-j},d-j} - \frac{\varepsilon_{d-j}}{\varepsilon_{d-j-1}'} \beta^{p^{d-j}(p-1)}f_{m+p^{d-j-1},d-j-1} \\
	        &= \frac{1}{p^{\val_p(m+p^{d-j})-d+j}}D^0(y^{(m+p^{d-j})/p}) + (\text{integral element}).
	    \end{align}
	    We divide into $3$ cases.
	    
	    Case 1-1: If $\val_p(m+p^{d-j})=d-j$, then the whole right-hand side of \eqref{eq:extb} is integral. In this case, we see that
	    \begin{align*}
	        pg_{m+p^{d-j},d-j} - \frac{\varepsilon_{d-j}}{\varepsilon_{d-j-1}'} g_{m+p^{d-j-1},d-j-1}
	    \end{align*}
	    is integral so that by taking $D^1$, we have
	    \[
	        ph_{m,j} = h_{m,j+1}
	    \]
	    in $\Ext^2_{\Gamma_\ast}(A_\ast)$ up to a $p$-adic unit.
	    
	    Case 1-2: If $m=p^d(p-1)$ and $j=0$, then from the proof of Theorem \ref{thm:bbstn}.(a), 
	    \[
	        \frac{D^0(y^{p^d})}{p} = \delta_d\beta^{p^{d+2}-p^{d+1}}\frac{\sigma^2 t_{e+1}}p + \text{(integral element)}.
	    \]
	    Therefore,
	    \[
	        pg_{p^{d+1},d} - \frac{\varepsilon_{d}}{\varepsilon_{d-1}'} g_{m+p^{d-1},d-1} - \delta_d\beta^{p^{d+2}-2p^{d+1}+2}\frac{\sigma^2 t_{e+1}}{p}
	    \]
	    is integral, and by taking $D^1$, we obtain
	    \[
	        ph_{p^d(p-1),0} = h_{p^d(p-1),1}.
	    \]
	    
	    Case 1-3: The remaining case is when $\val_p(m+p^{d-j})>d-j$ but not $(m,j)\neq(p^d(p-1),0)$. This can happen only if $j=0$. In this case, we have
	    \[
	        f_{m+p^d,d+1} = \frac{1}{p^{\val_p(m+p^d)-d}}D^0(y^{(m+p^d)/p}) + \beta^{p^{d+2}-p^{d+1}}\times\text{(cycle)}
	    \]
	    where (cycle) means some cycle in $A_\ast\otimes\Gamma_\ast[p^{-1}]$. Comparing with \eqref{eq:extb}, we see that
	    \[
	        pg_{m+p^d,d} - \frac{\varepsilon_{d}}{\varepsilon_{d-1}'}g_{m+p^{d-1},d-1} - \beta^{p^{d+2}-p^{d+1}} g_{m+p^d,d+1} + \beta^{p^{d+2}-2p^{d+1}+2} \times\text{(cycle)}
	    \]
	    is integral. Taking $D^1$, we have
	    \[
	        ph_{m,0} = h_{m,1} + \beta^{p^{d+2}-p^{d+1}}h_{m-p^d(p-1),d'-d-1}
	    \]
	    where $d'=\val_p(m-p^d(p-1))$. This completes the proof when $j<d-1$.
	    
	    Case 2: If $j=d-1$, then we have
	    \begin{align*}
	        &pf_{m+p,1} - c\beta^{p(p-1)}f_{m+1,0} \\
	        &= \frac{1}{p^{\val_p(m+p)-1}} D^0(y^{(m+p)/p}) \\
	        &\hphantom{==}+ \frac1p\beta^{p(p-1)}\left(\varepsilon_1D^0(y^{m/p}\sigma^2 x_{p-1})-cD^0((\sigma^2 x_{p-1})^{m+1})\right)
	    \end{align*}
	    and we can choose a $p$-adic unit $c$ so that
	    \[
	    \frac1p\beta^{p(p-1)}\left(\varepsilon_1D^0(y^{m/p}\sigma^2 x_{p-1})-cD^0((\sigma^2 x_{p-1})^{m+1})\right)
	    \]
	    is $\beta^{2p^2-3p}$ times a cycle by the definition of $y$. The rest of the argument is similar to the previous cases.
	    
	    Case 3: For $j=d$, we have
	    \[
	        pf_{m+1,0} = \frac{1}{p^{\val_p(m+1)}}D^0((\sigma^2 x_{p-1})^{m+1})
	    \]
	    and the rest of the argument is similar to the previous cases.
	    
	    This completes the proof of the description of $\Ext_{\Gamma_\ast}(A_\ast)$. The rest of the statement immediately follows. 
	\end{proof}
	
	\section{Integral Homotopy via Fracture Square}
	
	In this section, we shall compute the homotopy groups of $\THH(\ku)$, thereby proving \ref{thm:main}, by assembling the $p$-complete homotopy groups of $\THH(\ku)$. Recall that $\THH(\ku)$ splits into a direct sum
	\[
	    \THH(\ku) = \ku \oplus \overline{\THH(\ku)}.
	\]
	Let $X=\overline{\THH(\ku)}$. Then, we can compute the homotopy groups of $X$ using the arithmetic fracture square
	\[
	    \begin{tikzcd}
	        X\ar[r]\ar[d]&\prod_{p}X_p^\wedge\ar[d]\\
	        X\otimes\mathbb Q\ar[r]&\left(\prod_{p}X_p^\wedge\right)\otimes\mathbb Q.
	    \end{tikzcd}
	\]
	which is a pullback. Equivalently, there is a cofiber sequence
	\[
	    X\to (X\otimes\mathbb Q)\oplus \left(\prod_{p}X_p^\wedge\right) \to \left(\prod_{p}X_p^\wedge\right)\otimes\mathbb Q.
	\]
	
	\begin{lemma}
	    The rational homotopy groups of $\THH(\ku)$ are
	    \[
	        \pi_\ast\left(\THH(\ku)\otimes\mathbb Q\right) = \mathbb Q[\beta]\oplus\Sigma^3\mathbb Q[\beta]
	    \]
	    as $\mathbb Q[\beta]$-modules. The second summand is generated by $\sigma b_1$ where $b_1$ denotes the image of $b_1\in\MU_\ast\MU$ in $\ku_\ast\ku$.
	\end{lemma}
	\begin{proof}
	    The $E_2$-page of the B\"okstedt spectral sequence
	    \[
	        E_2 = \operatorname{HH}_\ast(\pi_\ast(\ku\otimes\mathbb Q)/\mathbb Q) \Rightarrow \pi_\ast\left(\THH(\ku)\otimes\mathbb Q\right)
	    \]
	    is
	    \[
	        E_2 = \Lambda_{\mathbb Q[\beta]}(\sigma(\beta_1-\beta_2))
	    \]
	    where $\beta_1,\beta_2$ are the two copies of $\beta$ in $\mathbb Q[\beta]\otimes_{\mathbb Q}\mathbb Q[\beta]$. The spectral sequence degenerates since there is no room for any differential, and the conclusion follows from the observation that $\sigma(\beta_1-\beta_2)$ detects $\sigma b_1$ rationally.
	\end{proof}
	
	\begin{proof}[Proof of Theorem \ref{thm:main}]
	    We shall prove that the map
	    \begin{equation}\label{eq:frac}
	         (X\otimes\mathbb Q)\oplus \left(\prod_{p}X_p^\wedge\right) \to \left(\prod_{p}X_p^\wedge\right)\otimes\mathbb Q
	    \end{equation}
	    in the arithmetic fracture square is a surjection on homotopy groups and compute the kernel.
	    
	    At even degrees, \eqref{eq:frac} is
	    \[
	        0\oplus\bigoplus_p T(p) \to 0
	    \]
	    so that the even homotopy group of $X$ is $\bigoplus_pT(p)$.
	    
	    Next, we note that in the expression
	    \[
	        \THH_\ast(\ku)_p^\wedge = \mathbb Z_p[\beta]\oplus F_p^\wedge\oplus T(p)
	    \]
	    of Theorem \ref{thm:mainp}, the generator of $\THH_3(\ku)_p^\wedge$, i.e. the lowest degree generator of $F_p^\wedge$, is $\sigma b_1$. This follows from the proof of Theorem \ref{thm:mainp} in Section \ref{ssc:betaep} combined with Remark \ref{rmk:desus}.
	    
	    Then, at an odd degree, say $2k+3$, \eqref{eq:frac} is
	    \[
	        \mathbb Q \oplus \prod_p\mathbb Z_p\to\left(\prod_p\mathbb Z_p\right)\otimes\mathbb Q
	    \]
	    where $\mathbb Q$ on the left-hand side is generated by $\beta^k\sigma b_1$ and the $\mathbb Z_p$'s on both sides are generated by $(\beta^k/f(k))\sigma b_1$. This is the fracture square for the ordinary ring $\mathbb Q$. Therefore, it is surjective and the kernel is $\mathbb Z$ generated by $(\beta^k/f(k))\sigma b_1$.
	\end{proof}
	
	\section{Further Questions}\label{sec:FQ}
	    It would be interesting if the descent spectral sequence for $\THH(R)\to\THH(R/\MU)$ degenerates for more general $R$. When $R=\BP\langle n\rangle$ is a $\mathbb E_3$-$\MU$-algebra, \cite[Prop. 6.1.6]{HW} implies that $E_2^{s,t}(\THH(R))$ is concentrated in $0\leq s\leq n+1$, so that the degeneracy is not immediate as in Corollary \ref{cor:deg}. However, since the descent spectral sequence for $R=\MU$ degenerates at the $E_2$-page, the following conjecture would imply the degeneracy for $R=\BP\langle n\rangle$. Note that the assumption that $R$ is a $\mathbb E_3$-$\MU$-algebra is needed to ensure that $\THH_\ast(R/\MU)$ is a commutative ring.
	    
        \begin{conjecture}
            Suppose that we have $\BP\langle n\rangle$ with an $\mathbb E_3$-$\MU$-algebra structure, which exists by \cite{HW}. Let $E_2(\THH(\MU))$ and $E_2(\THH(\BP\langle n\rangle))$ be the $E_2$-page of the descent spectral sequence for $\THH(\MU)\to\THH(\MU/\MU)$ and $\THH(\BP\langle n\rangle)\to\THH(\BP\langle n\rangle/\MU)$. Then, 
            \[
            E_2^{s,t}(\THH(\MU))\to E_2^{s,t}(\THH(\BP\langle n\rangle))
            \]
            is surjective for $0\leq s\leq n$.
        \end{conjecture}
        
        Theorem \ref{thm:bbstn}(a) shows that the conjecture is true for $\ku$ instead of $\BP\langle 1\rangle$, i.e. the map $E_2^{s,t}(\THH(\MU))\to E_2^{s,t}(\THH(\ku))$ is surjective for $0\leq s\leq 1$. Similar computations can be done to show that the conjecture is true for $0\leq s\leq 1$ and any $n$.
    \bibliographystyle{amsalpha}
    \bibliography{ref}
\end{document}